\documentclass[12pt,reqno]{amsart}
\usepackage{times}
\usepackage{amsmath,amssymb,amsthm,mathtools,enumerate}
\usepackage{datetime,graphicx}
\usepackage[usenames,dvipsnames]{color}
\usepackage{dsfont}
\usepackage{hyperref}
\hypersetup{pdfcreationdate=\pdfdate,pdfmoddate=\pdfdate,colorlinks=true,citecolor=BlueViolet,linkcolor=BlueViolet,urlcolor=BlueViolet}
\usepackage[margin=1in]{geometry}

\numberwithin{equation}{section}

\newtheorem{theorem}[equation]{Theorem}
\newtheorem{prop}[equation]{Proposition}

\newtheorem*{mainthm}{Main Theorem}
\newtheorem{lemma}[equation]{Lemma}
\newtheorem{cor}[equation]{Corollary}
\theoremstyle{definition}
\newtheorem{definition}[equation]{Definition}
\newtheorem{example}[equation]{Example}

\def\bysame{\leavevmode\hbox to3em{\hrulefill}\thinspace}

\newcommand{\abs}[1]{\left\lvert#1\right\rvert}

\renewcommand{\Re}{\operatorname{Re}}
\renewcommand{\Im}{\operatorname{Im}}

\newcommand{\nuP}{\boldsymbol{\nu}_{_{\!P\!}}}
\newcommand{\CTP}{\mathcal{T}_{_{\!P\!}}}
\newcommand{\UCTP}{\mathds{T}_{_{\!P\!}}}
\newcommand{\UCTO}{\mathds{T}_{_{0}}}

\DeclareMathOperator{\Lip}{\operatorname{Lip}}

\DeclareMathOperator{\dist}{\operatorname{dist}}

\DeclareMathOperator{\RTyp}{\Delta_{\text{reg}}}

\DeclareMathOperator{\ord}{\operatorname{ord}}

\begin{document}

\title{Tangential Lipschitz gain for holomorphic functions on domains of finite type}
\author[R. Sivaguru]{Sivaguru Ravisankar}
\thanks{Research partially supported by The James D. Wolfensohn Fund.}
\address{Department of Mathematics, Oklahoma State University, Stillwater, Oklahoma 74078.}
\email{\href{mailto:sivagr@okstate.edu}{\nolinkurl{sivagr@okstate.edu}}}
\subjclass[2010]{32A37, 32F32, 26A16}
\date{\usdate\today}

\begin{abstract}
Functions that are holomorphic and Lipschitz in a smoothly bounded domain enjoy a gain in the order of Lipschitz regularity in the complex tangential directions near the boundary. We describe this gain explicitly in terms of the defining function near points of finite type in the boundary.
\end{abstract}

\maketitle

\section{Introduction}

The remarkable phenomenon of tangential Lipschitz gain for holomorphic functions was discovered by E. Stein \cite{Ste73}.
We begin by recalling his foundational result which states that a Lipschitz-$\alpha$ holomorphic function is Lipshchitz-$2\alpha$ along complex tangential curves near the boundary.

Let $\Omega$ be a smoothly bounded domain (i.e., bounded domain with $C^\infty$ boundary) in $\mathbb{C}^n$, $n>1$.
For $P\in b\Omega$, let $\CTP(b\Omega)$ denote the complex tangent space to $b\Omega$ at $P$.
Let $\mathcal{O}(\Omega)$ denote the set of holomorphic functions defined in $\Omega$.
For a domain $D$ in Euclidean space and 
$0<\alpha<1$, let $\Lip_\alpha(D)$ denote the Lipschitz (or H\"{o}lder) space of order $\alpha$. i.e., 
\[\Lip_\alpha(D) = \left\{f:D\rightarrow \mathbb{C}\,:\, \exists\,C_f>0, \abs{f(x)-f(y)} \le C_f \cdot \abs{x-y}^\alpha \text{ for } x,y\in D\right\}.\]

\begin{theorem}[Stein \cite{Ste73}] Suppose $f\in\mathcal{O}(\Omega)\cap C(\overline{\Omega})$.
If $f\in\Lip_\alpha(\Omega)$ for some $0<\alpha < 1/2$, then $\exists\, C>0$ such that, for any normalized  complex tangential curve $\gamma :  [0,1] \rightarrow \Omega$ and $s,t\in[0,1]$, we have
\[\abs{f\circ\gamma(s) - f\circ\gamma(t)}\le C\cdot\abs{s-t}^{2\alpha}.\]
A curve $\gamma$ is called normalized complex tangential if $\abs{\gamma'(t)} \le 1$ and $\gamma'(t) \in \mathcal{T}_{\pi(\gamma(t))}(b\Omega)$ where $\pi$ is a normal projection from a neighbourhood of $b\Omega$ to $b\Omega$.
\end{theorem}
This result is a striking example of the deep interplay between the function theory of holomorphic functions and boundary geometry 
in several complex variables. The key geometric fact driving this result is the following: for a point $z\in\Omega$ near 
$b\Omega$, a complex disc centred at  $z$ with radius $\approx{\dist_{b\Omega}(z)}^{1/2}$ can be fit inside the domain in any 
complex tangential direction.
This translates to a better estimate for the rate of growth of the complex tangential derivative of Lipschitz 
holomorphic functions which in turn results in a gain in Lipschitz regularity in these directions.
The Hardy-Littlewood theorem (Theorem \ref{thm:H-L}) provides the clearest indication of the close connection that is shared between the Lipschitz regularity of a holomorphic function and the rate of growth of its derivative.

We should note that the tangential Lipschitz gain phenomenon is local in nature and the Lipschitz gain reflects the local boundary geometry near a point on the boundary.
Also, the restriction of $\alpha < 1/2$ in Stein's result is primarily to avoid the technicalities of dealing with Lipschitz spaces of order 1 or greater.
For more on such matters see the expository article by Krantz \cite{Kra83}.

Rudin \cite{Rud78} showed a similar gain of Lipschitz-$\alpha$ to Lipschitz-$2\alpha$ in the unit ball under weaker hypothesis and Krantz \cite{Kra80} has generalized Rudin's result suitably to show the same gain on any smoothly bounded domain.
These results that show a gain from Lipschitz-$\alpha$ to Lipschitz-$2\alpha$ are sharp only for strongly pseudoconvex domains.
A better gain is expected in the weakly pseudoconvex tangential directions since discs of larger relative radii can be fit in these directions and this translates to an even better estimate of the complex tangential derivatives in these directions.
Exploring this aspect of the tangential Lipschitz gain phenomenon near points of finite type is the central theme of this paper.
We will show that the gain is related to the non-isotropic nature of the complex geometry of the boundary measured by the relative size of discs that can be fit in the complex tangential directions. Since holomorphic functions extend holomorphically past boundary points that are not pseudoconvex, it is natural to restrict our attention to pseudoconvex domains. Furthermore, we will study one complex tangential direction at a time and formulate the Lipschitz gain along complex discs in such a direction.

Results that show a Lipschitz gain better than Lipschitz-$\alpha$ to Lipschitz-$2\alpha$ are known for finite type domains in $\mathbb{C}^2$ and convex domains of finite type in any dimension.
Krantz \cite{Kra90} has formulated this gain on any smoothly bounded domain using an Eisenman-Kobayashi like metric and volume. 
Chang and Krantz \cite{Cha-Kra91} have obtained an effective version of this result on pseudoconvex domains of finite type in $\mathbb{C}^2$ showing that the gain is determined by the type of the boundary point, i.e., a gain of Lipschitz-$\alpha$ to Lipschitz-$k\alpha$ near a point of type $k$.
This can be achieved either by using the estimates of invariant metrics on such domains by Catlin \cite{Cat89} or by using the works of Nagel, Stein, and Wainger \cite{NSW81}, McNeal \cite{McN89}, and Nagel, Rosay, Stein, and Wainger \cite{NRSW89}.
McNeal and Stein \cite{McN-Ste94} have formulated this gain on convex domains of finite type in terms of  the size of certain natural polydiscs. These polydiscs, constructed by 
McNeal \cite{McN94}, precisely capture the non-isotropic nature of the boundary of convex domains.
Our result is in a similar vein but we consider domains of finite type in any dimension that are not necessarily convex. 
We need the following definitions to state our main result.

For $P\in b\Omega$, let $\nuP$ denote the outward unit normal of $b\Omega$ at $P$.
A real valued function $r$ is a defining function of $\Omega$ near $P$ if it is defined in a neighbourhood $U$ of $P$ and it satisfies the following conditions:
\[\Omega\cap U = \left\{z\in\mathbb{C}^n : r(z) < 0\right\},\ b\Omega\cap U = \left\{z\in\mathbb{C}^n : r(z) = 0\right\}, \text{ and } \nabla r \neq 0 \text{ on } b\Omega.\]
Let $\UCTP(b\Omega)$ denote the unit vectors in the complex tangent space $\CTP(b\Omega)$.
\begin{definition}\label{defn:RS}
For $P\in b\Omega$, a defining function $r$ of $\Omega$ near $P$, $\boldsymbol{v}\in\UCTP(b\Omega)$, small $t>0$, and a choice of coordinates near $P$, we define the following.
\begin{align*}
P_t &:= P - t\nuP,\\
R(t) & := \sup\left\{ \rho > 0 \,:\, P_t + \zeta\boldsymbol{v} \in \Omega\cap U \text{ for } \zeta\in\mathbb{C}, \abs{\zeta} \le \rho  \right\},\\ 
\mathbb{D}(t) & := \left\{P_t + \zeta\boldsymbol{v} \,:\, \zeta\in\mathbb{C}, \abs{\zeta} <  {R}(t)  \right\}, \text{ and }\\ 
S(t) & := \sup\left\{r\left(P+\zeta\boldsymbol{v}\right) \,:\, \zeta\in\mathbb{C},\, \abs{\zeta}\:\le\: t\right\}. &\phantom{\hspace{\linewidth}}
\end{align*}
\vspace*{-5.25\baselineskip}
\begin{figure}[h!]
\flushright
\begin{minipage}[t]{8.02cm} 
\centering
\includegraphics[height=2.6385cm]{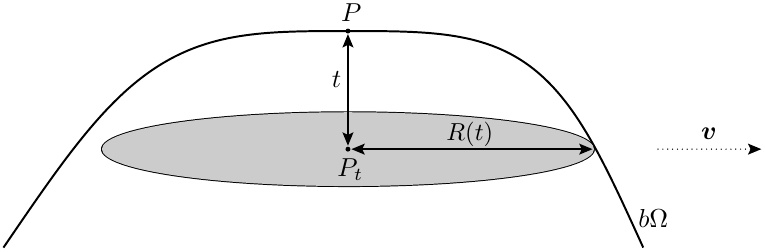}
\caption{Function $R$ and Disc $\mathbb{D}$.}
\end{minipage}
\end{figure}
\end{definition}

A few remarks are in order. First, note that, in the interest of readability, we have suppressed the possible dependence of $R$, $\mathbb{D}$, and $S$ on $P$, $\boldsymbol{v}$, the choice of coordinates, and the defining function in our choice of notation.
However, we will highlight and comment on this dependence in our estimates and results whenever it is relevant. Second, $S(t) \le({const})\cdot t^2$ for any pseudoconvex domain $\Omega$, $P\in b\Omega$, and $\boldsymbol{v}\in \UCTP(b\Omega)$.
But, the function $S$ satisfies a possibly better estimate when $\boldsymbol{v}$ corresponds to a weakly pseudoconvex direction and $b\Omega$ has a finite order of `holomorphic flatness' at $P$ (see  Proposition \ref{prop:RS_Props}). i.e., $S(t) \le({const})\cdot t^{k}$, for some even integer $k$ that is less than or equal to this order of `holomorphic flatness'.
In such a setting, $S$ is also the inverse function of $R$.

We use the notion of type introduced by D'Angelo \cite{D'A82} to measure the `holomorphic flatness' of $b\Omega$ at a boundary point.
The type of a point $P\in b\Omega$, denoted by $\Delta(b\Omega,P)$, is the maximum order of contact of holomorphic curves with $b\Omega$ at $P$ (see Section \ref{subsec:FinTyp} for the definition). 
A related construct, the regular type, denoted by $\RTyp(b\Omega,P)$, is the maximum order of contact of non-singular holomorphic curves with $b\Omega$ at $P$. We are now ready to state our main theorem.

\begin{mainthm} Let $\Omega$ be a smoothly bounded pseudoconvex domain in $\mathbb{C}^n$, $n>1$, and let $P\in b\Omega$ be a point of finite regular type.
If $f\in\mathcal{O}(\Omega)\cap\Lip_\alpha(\Omega)$ for $0 < \alpha < 1/\RTyp\left(b\Omega, P\right)$, then $\exists\, C>0$ such that 
\[\abs{f\!\left(P_\delta\right) - f(w)} \, \le\, C\cdot S(\abs{P_\delta-w})^\alpha\]
for $v\in\mathds{T}_{P}(b\Omega)$, small $\delta>0$, and $w\in\mathbb{D}(\delta)$.
\end{mainthm}
The constant $C$  depends on the local geometry of $b\Omega$ near $P$ and can be chosen to be independent of $\boldsymbol v$ and the choice of coordinates.
The only dependence of $C$ on $f$ is on the Lipschitz-$\alpha$ norm of $f$.
Furthermore, if the regular type of $b\Omega$ is bounded in a neighbourhood of $P$, we can choose $C$ to be independent of $P$ as well. Notice however that $S$ depends on $P$, $\boldsymbol v$, the choice of coordinates, and the defining function.
The dependence of $S$ on the defining function is a red herring since the quotient of two smooth defining functions is positive near $b\Omega$.

D'Angelo also showed that finite type is an open condition, with an explicit bound, whereas finite regular type need not be. This gives us the following succinct corollary. We should bear in mind that the result holds locally near a point of finite type as well. 
\begin{cor} Let $\Omega$ be a smoothly bounded pseudoconvex domain of finite type in $\mathbb{C}^n$, $n>1$.
Then, there exists $\alpha_0, C > 0$ such that, if $f\in\mathcal{O}(\Omega)\cap\Lip_\alpha(\Omega)$ for $0 < \alpha < \alpha_0$, then 
\[\abs{f\!\left(P_\delta\right) - f(w)} \, \le\, C\cdot S(\abs{P_\delta-w})^\alpha\]
for $P\in b\Omega$, ${\boldsymbol v}\in\UCTP(b\Omega)$, small $\delta>0$,  and $w\in\mathbb{D}(\delta)$.
\end{cor}

We can generalize this result further by using a result of the author \cite{Rav13} which states that a harmonic function that is Lipschitz along a family of transversal curves near the boundary is Lipschitz in all directions. We state this as a corollary and it is analogous to a result of Krantz \cite{Kra80}.

\begin{cor}Let $\Omega$ be a smoothly bounded pseudoconvex domain of finite type in $\mathbb{C}^n$, $n>1$.
Then, there exists $\alpha_0, C>0$ such that, if $f$ is a bounded holomorphic function in $\Omega$ and uniformly Lipschitz-$\alpha$ along a family of curves transversal to $b\Omega$, for $0 < \alpha < \alpha_0$, then 
\[\abs{f\!\left(P_\delta\right) - f(w)} \, \le\, C\cdot S(\abs{P_\delta-w})^\alpha\]
for $P\in b\Omega$, ${\boldsymbol v}\in\UCTP(b\Omega)$, small $\delta>0$,  and $w\in\mathbb{D}(\delta)$.
\end{cor}

The article is organized as follows. Some interesting examples that highlight the features and limitations of our result are presented in Section \ref{sec:Examples}.
We collect some key results concerning Lipschitz functions and facts about finite type domains in Section \ref{sec:Prelims}. Section \ref{sec:MainThm} is devoted to stating and proving the main result of this article.  We shall use $A\,\lesssim\,B$ or $B\,\gtrsim\,A$ to mean $A \le 
C\cdot B$ for some constant $C>0$ which is independent of certain parameters. It will be mentioned, or clear from the context, what these parameters are. We use $A \approx B$ to mean $A\,\lesssim\,B$ and $B\,\lesssim\,A$.

\section{Examples}\label{sec:Examples}
Let us begin by using our result to recover a version of Stein's original result of Lipschitz-$\alpha$ to Lipschitz-$2\alpha$ gain for holomorphic functions.
\begin{example}[Pseudoconvex Domains] Let $\Omega$ be a smoothly bounded psuedoconvex domain. 
Then, $S(t) \,\lesssim\, t^2$ for any $P\in b\Omega$ and ${\boldsymbol v}\in\UCTP(b\Omega)$.
Notice that this ignores any information about the weakly pseudoconvex directions that $b\Omega$ may have at $P$. Hence, if $f\in\mathcal{O}(\Omega)\cap\Lip_{\alpha}(\Omega)$, for some $0<\alpha<1/2$, then
\[\abs{f(P_\delta) - f(w)} \,\lesssim\, \abs{P_\delta - w}^{2\alpha}\]
for $P\in b\Omega$, ${\boldsymbol v}\in\UCTP(b\Omega)$, small $\delta>0$,  and $w\in\mathbb{D}(\delta)$.
\end{example}

\begin{example}[Herbort's Domain] Herbort \cite{Her83} constructed the following domain in $\mathbb{C}^3$ to show that the Bergman kernel function need not grow like a rational power of the distance to the boundary function.
Our results have been inspired, in large part, by trying to gain a deeper understanding of the tangential Lipschitz gain phenomenon on this domain.
This domain $\Omega$ is defined near the origin by
\[r(z) = \Re z_3 + \abs{z_1}^6 + \abs{z_1z_2}^2 + \abs{z_2}^6.\]
It is a finite type domain but, crucially, is not convex. The Lipschitz gain on discs centred below the origin is given by
\[S(t) = \abs{v_1}^6 t^6 + \abs{v_1v_2}^2 t^4 +\abs{v_2}^6 t^6,\]
for ${\boldsymbol v}=(v_1,v_2,0)\in \UCTO(b\Omega)$. If $f\in\mathcal{O}(\Omega)\cap\Lip_\alpha(\Omega)$, for $0<\alpha<1/6$, then
\[\abs{f(0,0,-\delta) - f(\zeta v_1, \zeta v_2,-\delta)} \,\lesssim\, \left(\abs{\zeta}^6\abs{v_1}^6 + \abs{\zeta}^4\abs{v_1v_2}^2 + \abs{\zeta}^6\abs{v_2}^6\right)^\alpha\]
for small $\delta > 0$, ${\boldsymbol v}\in\UCTO(b\Omega)$, and $(\zeta v_1, \zeta v_2,-\delta)\in\mathbb{D}(\delta)$.

Note that $\Delta(0,b\Omega) = 6$ and we get the expected gain of Lipschitz-$\alpha$ to Lipschitz-$6\alpha$ in the $(1,0,0)$ and $(0,1,0)$ directions but we only get a gain of Lipschitz-$\alpha$ to Lipschitz-$4\alpha$ in the $(1,1,0)$ direction!

Since finite type is an open condition, the Lipschitz gain on discs centred below a point $P\in b\Omega$, close to $0$, in a direction ${\boldsymbol v}\in\UCTP(b\Omega)$ is given by, a different function, $S$ based at $P$ which doesn't have as nice an expression as above.
\end{example}

Lastly, let us examine a domain, constructed by D'Angelo \cite{D'A82}, that highlights the subtlety between type and regular type and also underscores the limitations of our result.

\begin{example}\label{eg:Type} Consider the domain $\Omega$ defined near the origin by
\[r(z) = \Re z_3 +\abs{z_1^2-z_2^3}^2.\]
This domain has the property that $0\in b\Omega$ is a point of finite regular type but not a point of finite type.
Furthermore, there are points arbitrarily close to the origin that have infinite regular type showing, in particular, that finite regular type is not an open condition. For more on this, see Section \ref{subsec:FinTyp}.

On this domain our result applies only to discs centred below the origin, on which we get a gain of Lipschitz-$\alpha$ to Lipschitz-$6\alpha$, but does not apply even to boundary points arbitrarily close to the origin. We wish to extend our understanding of the Lipschitz gain phenomenon further so as to enable us to describe it in a full neighbourhood of the origin on this domain.
\end{example}

\section{Preliminaries}\label{sec:Prelims}

\subsection{Lipschitz Functions}

The following theorem of Hardy and Littlewood provides an insightful characterization of when a harmonic function is Lipschitz in terms of the rate of growth of its derivative.

\begin{theorem}[Hardy-Littlewood]\label{thm:H-L}Let $\Omega$ be a smoothly bounded domain in $\mathbb{R}^m$.  Let $U$ be a neighbourhood of $b\Omega$ and $0<\alpha<1$. Suppose $u$ is harmonic and bounded in $\Omega$. Then, $u\in\Lip_\alpha(\Omega)$ if and only if $u$ satisfies
\begin{equation}\label{eqn:HL_Ineq}\abs{\nabla u(x)} \,\lesssim\, \dist_{b\Omega}(x)^{-1+\alpha}, \ \text{ for }\ x \in U\cap \Omega.\end{equation}
\end{theorem}
\begin{proof} See \cite[Theorem 2.4 and Lemma 2.6]{Rav13}. Furthermore, when $u$ is harmonic and Lipschitz, the constant in the inequality above can be chosen so that its only dependence on $u$ is on $\left\| u\right\|_\alpha$\footnote{$\displaystyle\left\| u\right\|_{\alpha} = \left\|u\right\|_\infty + \sup\limits_{\substack{x,y\in\Omega\\ x\ne y}}\, \dfrac{\abs{u(x) - u(y)}}{\abs{x-y}^\alpha}.$}.
\qed\end{proof}

The following lemma shows that the derivatives of a Lipschitz holomorphic function can be estimated in terms of the size of certain discs that can be fit inside the domain.
In conjunction with the Hardy-Littlewood theorem, these estimates are at the very heart of the tangential Lipschitz gain phenomenon and provide a bridge between boundary geometry and function theory of Lipschitz holomorphic functions.

\begin{lemma}\label{lemma:DerivLipHolo} Let $\Omega$ be a smoothly bounded domain in $\mathbb{C}^n$, $n>1$, and $f\in\mathcal{O}(\Omega)\cap\Lip_\alpha(\Omega)$ for some $0<\alpha<1$. Suppose $\boldsymbol{n}$ and $\boldsymbol{v}$ are unit vectors in $\mathbb{C}^n$ and $z\in\Omega$. Then, we have the following.
\begin{enumerate}[(i)]\itemsep5mm
   \item\label{lemma:DerivLipHoloA} If $\left\{z+\zeta\boldsymbol{n} \, :\, \zeta\in\mathbb{C}, \abs{\zeta} \le \delta\right\} \subset \Omega$, for some $\delta>0$, then 
   \[\abs{\frac{\partial f}{\partial \boldsymbol{n}} (z)}\, \le\, \left\|f\right\|_{\alpha} \cdot\delta^{-1+\alpha}.\]
   \item\label{lemma:DerivLipHoloB} If $\left\{z+\zeta\boldsymbol{n} + \tau\boldsymbol{v} \, :\, \zeta,\tau\in\mathbb{C},\, \abs{\zeta} \le \delta, \abs{\tau} \le r\right\} \subset \Omega$, for some $\delta,r>0$, then 
\[\abs{\frac{\partial^2 f}{\partial \boldsymbol{n}\, \partial\boldsymbol{v}} (z)}\, \le\, \left\|f\right\|_{\alpha}\cdot r^{-1}\cdot\delta^{-1+\alpha}.\]
\end{enumerate}
\end{lemma}
\begin{proof} These are direct consequences of Cauchy's integral formula. Without loss of generality let us suppose that $\left\| f\right\|_\alpha \le 1$.
\begin{enumerate}[(i)]
\item\begin{minipage}[t]{\linewidth}\vspace*{-1\baselineskip}
\begin{align*}
\frac{\partial f}{\partial \boldsymbol{n}} (z) &= \frac{1}{2\pi i}\int\limits_{\abs{\zeta} = \delta}\, \frac{f(z+\zeta\boldsymbol{n})}{\zeta^2}\, \text{d}\zeta = \frac{1}{2\pi i}\int\limits_{\abs{\zeta} = \delta}\, \frac{f(z+\zeta\boldsymbol{n})-f(z)}{\zeta^2}\, \text{d}\zeta\\[3mm]
\abs{\frac{\partial f}{\partial \boldsymbol{n}} (z)} &\le \frac{1}{2\pi}\cdot \frac{\delta^\alpha}{\delta^2}\cdot 2\pi\delta = \delta^{-1+\alpha}.\\[5mm]
\end{align*}
\end{minipage}
\item \qquad$\displaystyle \frac{\partial^2 f}{\partial \boldsymbol{n}\, \partial\boldsymbol{v}} (z) = \frac{1}{2\pi i}\int\limits_{\abs{\tau} = r}\, \frac{1}{\tau^2}\cdot\frac{\partial f}{\partial \boldsymbol{n}}(z+\tau\boldsymbol{v})\, \text{d}\tau .$\\[2mm]
For $\abs{\tau}=r$, we can fit a disc of radius $\delta$ in the $\boldsymbol{n}$ direction centred at $z+\tau\boldsymbol{v}$ contained in $\Omega$. Hence, by \textit{(i)}, $\abs{\left(\partial f / \partial \boldsymbol{n}\right)(z+\tau\boldsymbol{v})} \le \delta^{-1+\alpha}$. So, 
\[\abs{\frac{\partial^2 f}{\partial \boldsymbol{n}\, \partial\boldsymbol{v}} (z)} \le \frac{1}{2\pi}\cdot \frac{\delta^{-1+\alpha}}{r^2}\cdot 2\pi r = r^{-1}\cdot\delta^{-1+\alpha}.\]
\end{enumerate}
\qed\end{proof}

We will normally use the above lemma along with the following geometrically evident fact: the distance to the boundary of a point obtained by moving transversally from a boundary point is comparable to the distance moved.
\begin{lemma}\label{lemma:TransDistEst}
Let $\boldsymbol{n}$ be a unit vector that is transverse to $b\Omega$ at $P$. i.e., $ \boldsymbol{n}\cdot\nuP \ne 0$.
We then have
		\[s \,\lesssim\, \dist_{b\Omega}\left(P-s\boldsymbol{n}\right) \le s, \text{ for small } s>0.\]
\end{lemma}
\begin{proof}
See \cite[Lemma 3.2]{Rav13}. The constant in the inequality on the left only depends on $\boldsymbol{n}\cdot\nuP$.
\qed\end{proof}

\subsection{Finite Type}\label{subsec:FinTyp}

The type of a boundary point (the 1-type to be precise), introduced by D'Angelo \cite{D'A82}, is the largest of order of contact of holomorphic curves to the boundary at that point.
It is a quantified measure of the degeneracy of the Levi form at that point.
If the boundary is strongly pseudoconvex at a point, then the type of that point is 2.
The type is greater at points where the boundary is only weakly pseudoconvex.

For a smooth function defined near $0\in\mathbb{C}$, let $\ord(f)$ denote the order of vanishing of $f-f(0)$. For a smooth vector valued function $F=(f_1,\ldots,f_n)$ defined near $0\in\mathbb{C}$, let $\ord(F) = \min\limits_j\, \ord(f_j)$.
Then, the type of $P\in b\Omega$ is defined as follows:
\[\Delta(b\Omega,P) = \sup\limits_F\, \frac{\ord(r\circ F)}{\ord(F)}\]
where the supremum is taken over holomorphic parametrizations $F$ of possibly singular holomorphic curves with $F(0)=P$. The regular type, denoted by $\RTyp(b\Omega,P)$, is defined similarly but the supremum is taken only over non-singular holomorphic curves.
A point $P\in b\Omega$ is said to be of {\it finite type} (respectively {\it finite regular type}) if $\Delta(b\Omega,P) < \infty$ (respectively $\RTyp(b\Omega,P) < \infty$). 

These definitions can be shown to be independent of the choice of defining function.
The subtle difference between the notion of type and regular type is best highlighted by the domain $\Omega$ discussed in Example \ref{eg:Type}. It is defined locally near the origin by the defining function $r(z)=\Re z_3 + \abs{z_1^2-z_2^3}^2$.
The holomorphic curve $\zeta \mapsto (\zeta^3,\zeta^2,0)$ lies in $b\Omega$, passes through the origin and is singular there.
Hence, we have $\Delta (b\Omega, 0) = \infty$.
Furthermore, since this curve is non-singular for $\zeta\neq 0$, $\RTyp(b\Omega, P)=\infty$ for points $P\in b\Omega$ that are arbitrarily close to $0$ that lie on this curve.
But, one can check that $\RTyp(b\Omega, 0)=6$.
This shows that finite regular type need not be an open condition.
The following theorem of D'Angelo \cite[Corollary 5.6]{D'A82} shows that finite type on the other hand {\it is} an open condition and also provides an explicit bound.
\begin{theorem}
Suppose $\Omega$ is a smoothly bounded pseudoconvex domain in $\mathbb{C}^n$ and $Q\in b\Omega$ is a point of finite type. Then, there is a neighbourhood $U$ of $Q$ such that
\[\Delta(b\Omega, P) \le \frac{\big(\Delta(b\Omega, Q)\big)^{n-1}}{2^{n-2}}, \text{ for } P\in U\cap b\Omega.\]
\end{theorem}
The following proposition from D'Angelo \cite[p. 138]{D'A93} will serve useful for us later. It states that the order of vanishing of a holomorphic curve at a point in the boundary of a pseudoconvex domain is even or infinite.
\begin{prop}\label{prop:EvenOrderVanish}
Let $\Omega$ be a smoothly bounded pseudoconvex domain in $\mathbb{C}^n$ with a local defining function $r$ near $P\in b\Omega$.
Suppose that $\gamma$ is a parametrized holomorphic curve such that $\gamma(0)=P$ and $\ord(r \circ \gamma ) = m < \infty$.
Then, the order of vanishing $m$ is even and the coefficient of $\abs{\zeta}^m$ in $r \circ \gamma$ is positive.
\end{prop}
Using the above result it can be seen that the type of a point in the boundary of a pseudoconvex domain is either even or infinite. This holds for regular type as well.

\section{Main Theorem}\label{sec:MainThm}

Let $\Omega$ be a smoothly bounded pseudoconvex domain in $\mathbb{C}^n$, $n>1$, with finite regular type at $P\in b\Omega$.
Suppose $r$ is a local defining function of $\Omega$ in a neighbourhood $U$ of $P$.
We make the following choices and normalizations.
Restrict $U$ if necessary so that we have a choice of coordinates $z$ in $U$ such that $z(P)=0$, the outward unit normal $\nuP(b\Omega)=(0,\ldots,0,1)$, and $b\Omega \cap U$ is the graph of a smooth function $-h$ defined in a neighbourhood of $0\in\mathbb{C}^{n-1} \times \mathbb{R}$ with $h(0)=0$ and $\nabla h (0) = 0$. i.e., $r(z) = \Re z_n + h(z_1,\ldots,z_{n-1},\Im z_n)$.

We begin by showing that the functions $R$ and $S$ (see Definition \ref{defn:RS}) are inverse functions and their behaviour is closely tied to the regular type of $b\Omega$ at $P$.
Since the Lipschitz gain is captured by the function $S$, we wish to draw the reader's attention to its relationship with $\RTyp(b\Omega, P)$.

\begin{prop}\label{prop:RS_Props} Let $\Omega$, $P$, $z$, and $\boldsymbol{v}$ be as above. Then,
\begin{enumerate}[(a)]\itemsep1ex
\item $R(t)$ and $S(t)$ are non-decreasing for small $t>0$,
\item\label{prop:RS_PropsB} $R(S(t))=t$ for small $t>0$, and
\item there exists a positive even integer $k_0$ such that the limits$\displaystyle\lim\limits_{\hphantom{^+}t \to 0^+}\, \frac{S(t)}{t^{k_0}}$ and, hence,$\displaystyle\lim\limits_{\hphantom{^+}t\to 0^+}\, \frac{R(t)}{t^{1/k_0}}$ exist and are positive.\smallskip
\end{enumerate}
\end{prop}
\begin{proof} Let $r$ be a local defining function of $\Omega$ near $P\in b\Omega$ that is normalized as above.
\begin{enumerate}[(a)]
\item $S$ is non-decreasing since $\displaystyle S(t) = \sup\left\{h\left(\zeta\boldsymbol{v}\right)\,:\, \zeta\in\mathbb{C},\,\abs{\zeta}\:\le\: t \right\}$ and $h(z',0) \ge 0$, for $z'$ near $0$, by Proposition \ref{prop:EvenOrderVanish}.
To conclude that $R$ is non-decreasing, notice that $\exists\,\rho_{_U} > 0$ such that for small $t>0$, 
\[R(t) = \sup\left\{0 < \rho < \rho_{_U} \,:\, S(\rho) < t\right\}.\]
\item Since $h$ is smooth, $S$ is continuous. Hence, (b) follows from the observation above.
\item For $\zeta\in\mathbb{C}$, $P+\zeta\boldsymbol{v}$ is a non-singular holomorphic curve.
Since $\RTyp(b\Omega,p) < \infty$, by Proposition \ref{prop:EvenOrderVanish}, there exists a positive even integer $k_0 \le  \RTyp(b\Omega,p)$ such that
\begin{align*}
\lim\limits_{\abs{\zeta}\to 0}\, \frac{r(P+\zeta\boldsymbol{v})}{\abs{\zeta}^k} = 0 \text{ for } 0\le k < k_0
\intertext{and}
\lim\limits_{\abs{\zeta}\to 0}\, \frac{r(P+\zeta\boldsymbol{v})}{\abs{\zeta}^{k_0}} \text{ exists and is positive.}
\end{align*}
Hence,$\displaystyle\lim\limits_{ \hphantom{^+}t\to 0^+}\, \frac{S(t)}{t^{k_0}}$ and$\displaystyle \lim\limits_{\hphantom{^+}t \to 0^+}\, \frac{R(t)}{t^{1/k_0}}$ exist and are positive.
\end{enumerate}
\qed\end{proof}
We should note that the $k_0$ in the above result depends on $P$, $z$, and $\boldsymbol{v}$. We now state and prove our main theorem.

\begin{mainthm} Let $\Omega$ be a smoothly bounded pseudoconvex domain in $\mathbb{C}^n$, $n>1$. Suppose $b\Omega$ has finite regular type at $P\in b\Omega$. If $f\in\mathcal{O}(\Omega)\cap\Lip_\alpha(\Omega)$ for $0 < \alpha < 1/\RTyp\left(b\Omega, P\right)$, then 
\[\abs{f\!\left(P_\delta\right) - f(w)} \, \lesssim\, S(\abs{P_\delta-w})^\alpha\]
for $v\in\mathds{T}_{P}(b\Omega)$, small $\delta>0$, and $w\in\mathbb{D}(\delta)$.
\end{mainthm}
Before we proceed to the proof we would like to note that the only dependence of the constant in the above inequality on $f$ is on $\left\|f\right\|_\alpha$. It also depends on the local geometry of $b\Omega$ near $P$ and can be chosen to be independent of $\boldsymbol v$ and the choice of coordinates.
Furthermore, if the regular type of $b\Omega$ is bounded in a neighbourhood of $P$, we can choose the constant to be independent of $P$ as well.
\begin{proof}
We push $P_\delta$ and $w$ away from the boundary and estimate the change of $f$ along the three edges of the rectangle formed by these four points (see Figure \ref{fig:BoxLipGain} below).
Let $h=S(\abs{P_\delta-w})$ be the distance by which $P_\delta$ and $w$ are pushed into $\Omega$ in the $\nuP$ direction.
Let $z_t=z-t\nu_P$ for $z$ near $P$ and small $t>0$.
Now, 
\[\abs{f(P_\delta)-f(w)} \le \underbrace{\strut \abs{f(P_\delta)-f(P_{\delta +h})}}_{\textit{I}} + \underbrace{\strut \abs{f(P_{\delta +h})-f(w_h)}}_{\textit{II}} + \underbrace{\strut \abs{f(w_h)-f(w)}}_{\textit{III}}.\]
The estimates for \textit{I} and \textit{III} are straightforward. The estimate for \textit{II} contains the interesting elements of the proof. Without loss of generality let us suppose that $\left\|f\right\|_{\alpha} \le 1$.

\begin{figure}[h!]
\centering 
\includegraphics[width=0.55\textwidth]{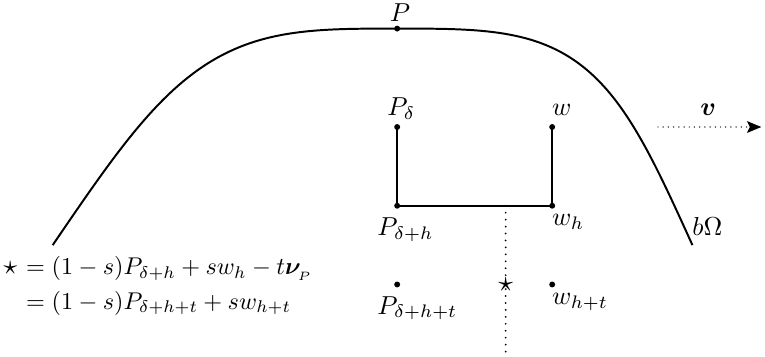}
\caption{Box argument - Lipschitz gain}
\label{fig:BoxLipGain}
\end{figure}

To estimate \textit{I}, notice that we have, by Lemma \ref{lemma:DerivLipHolo},
\begin{gather*}
\abs{\frac{\partial f}{\partial \nuP}(P_{\delta+s})} \le (\delta + s)^{-1+\alpha} \le s^{-1+\alpha},\ \text{ for small } s>0,
\intertext{and hence,}
\textit{I} = \abs{f(P_\delta)-f(P_{\delta +h})} \le \int\limits_0^h\, \abs{\frac{\partial f}{\partial \nuP}(P_{\delta+s})}\, \text{d}s \le \int\limits_0^h\, s^{-1+\alpha}\, \text{d}s \,\lesssim\, h^\alpha = S(\abs{P_\delta-w})^\alpha.
\end{gather*}
Since $\nuP$ is transversal to $b\Omega$ at $\pi(z)$, by Lemma \ref{lemma:TransDistEst}, we have $\dist_{b\Omega}(w_s) \,\gtrsim\, s$. So,
\[\abs{\frac{\partial f}{\partial \nuP}(w_s)} \,\lesssim\, s^{-1+\alpha}.\]
Note that the constant in the above inequality depends only on the geometry of $b\Omega$ near $P$. We estimate \textit{III} similarly to get
\[\displaystyle \textit{III} =  \abs{f(w_h)-f(w)} \,\lesssim\, S(\abs{P_\delta-w})^\alpha.\]

We estimate \textit{II} by using the mean value theorem.
A direct approach would involve estimating $\partial f/\partial \boldsymbol{v}$ using Lemma \ref{lemma:DerivLipHolo}\,\eqref{lemma:DerivLipHoloA}. But, this yields a weak estimate that is not quite sufficient for our purposes. To obtain a better estimate, we first estimate $\partial^2 f/(\partial \nuP\,\partial \boldsymbol{v})$ using Lemma 
\ref{lemma:DerivLipHolo}\,\eqref{lemma:DerivLipHoloB} and then integrate it in the $\nuP$ direction.
To execute this plan, we need good estimates on the size of discs that can be fit in the $\nuP$ and 
$\boldsymbol{v}$ directions at points on the dotted line in Figure \ref{fig:BoxLipGain}; a generic point being denoted by 
$\star=(1-s)P_{\delta+h+t} + sw_{h+t}$.

Let $0\le s \le 1$ and $t>0$ be sufficiently small. Since  $\dist_{b\Omega}(w_{h+t})\,\gtrsim\, h+t$  by Lemma \ref{lemma:TransDistEst}, we can fit a disc of radius $\,\gtrsim\, h+t$ centred at $\star$ in the $\nuP$ direction inside $\Omega$.
Also, a disc of radius $R(\delta+h+t) - \abs{P_\delta -w}$ centred at $\star$ can be fit inside $\Omega$ in the $\boldsymbol{v}$ direction.
To obtain the desired estimates, we now show that $R(\delta+h+t) - \abs{P_\delta - w} \,\gtrsim\, R(h+t)$.

By Proposition \ref{prop:RS_Props} we know that there exists a positive even integer $k_0 \le \RTyp(b\Omega,P)$ and $c>0$ such that$\displaystyle\lim\limits_{\hphantom{^+}u\to 0^+}\, {R(u)}/{u^{1/k_0}} = c$.
Choose $\eta > 0$ so that 
\[2\left(\frac{c-\eta}{c+\eta}\right) > 1+ \left(\frac{2}{3}\right)^{1/k_0}.\]
Then, $\exists\, \delta_0 > 0$ such that
\begin{equation}\label{eqn:R_Type}
c-\eta \,\le\, \frac{R(u)}{u^{1/k_0}} \,\le\, c + \eta \quad \text{ for } 0 < u \le 3\delta_0.
\end{equation}
Make $\delta_0$ smaller, if necessary, so that $z_{3\delta_0}\in\Omega$ for $z\in U\cap\Omega$ where $U$ is a sufficiently small neighbourhood of $P$.
Now, for $0 < v \le \delta_0$ and $v/2 \le x \le 2\delta_0$, we have
\[\frac{R(v+x) - R(v)}{R(x)} \ge \frac{(c-\eta)(v + x)^{1/k_0} - (c+\eta)v^{1/k_0}}{(c+\eta) x^{1/k_0}} = \left(\frac{c-\eta}{c+\eta}\right)\left(\frac{v}{x} + 1\right)^{1/k_0} - \left(\frac{v}{x}\right)^{1/k_0}.\]
Since $\displaystyle \left(\frac{c-\eta}{c+\eta}\right)\left(\bullet + 1\right)^{1/k_0} - \bullet^{1/k_0}$ is decreasing on $0\le \bullet \le 2$,\smallskip
\begin{equation}\label{eqn:R_Est}
\frac{R(v+x) - R(v)}{R(x)}  \ge \left(\frac{c-\eta}{c+\eta}\right)\cdot 3^{1/k_0} - 2^{1/k_0} > \frac{1}{2}\left(3^{1/k_0} - 2^{1/k_0}\right) >0.\smallskip
\end{equation}
By the hypothesis on $w$, $0 < \abs{P_\delta -w} < R(\delta)$. Since $h=S\left(\abs{P_\delta -w}\right)$ and $S$ is the inverse of $R$, we have $0 < h < \delta$.  So, $\exists\, n\in\mathbb{N}$ such that
\[\frac{\delta}{2^n} \,\le\, h \,\le\, \frac{\delta}{2^{n-1}} \,\le\, \delta_0 \quad\text{and}\quad \frac{\delta}{2^n} \,\le\, h + t \,\le\, 2\delta_0.\]
Hence,
\begin{align}
   R(\delta+h+t) - \abs{P_\delta-w} & \ge R(\delta+h+t) - R\left(\frac{\delta}{2^{n-1}}\right) \notag\\
   &\ge R\left(\frac{\delta}{2^{n-1}}+h+t\right) - R\left(\frac{\delta}{2^{n-1}}\right) \notag\\
   & \,\gtrsim\, R(h+t)  \qquad\qquad\quad\ (\text{by } \eqref{eqn:R_Est}).\label{eqn:WeakSuperAdd}
\end{align}
The constant in the above inequality depends only on the geometry of $b\Omega$ near $P$.
So, by Lemma \ref{lemma:DerivLipHolo},
\[\abs{\frac{\partial^2 f}{\partial \nuP\, \partial \boldsymbol{v}}\left(\star\right)} \,\lesssim\, \frac{(h+t)^\alpha}{(h+t)R(h+t)}.\]
Now, we integrate the above estimate in the $\nuP$ direction to obtain an estimate on $\partial f/\partial \boldsymbol{v}$.
\begin{multline}
\abs{\frac{\partial f}{\partial \boldsymbol{v}}\left((1-s)P_{\delta+h} + sw_{h}\right)} \le \abs{\frac{\partial f}{\partial \boldsymbol{v}}\left((1-s)P_{\delta+h+\delta_0} + sw_{h+\delta_0}\right)}\\
{+ \int\limits_0^{\delta_0}\, \abs{\frac{\partial^2 f}{\partial \nuP\, \partial \boldsymbol{v}}\left((1-s)P_{\delta+h+t} + sw_{h+t}\right)}\, \text{d}t} \label{eqn:IntInNuP}
\end{multline}
By Lemma \ref{lemma:DerivLipHolo} and \eqref{eqn:WeakSuperAdd}, we have
\begin{align*} 
&\abs{\frac{\partial f}{\partial \boldsymbol{v}}\left((1-s)P_{\delta+h+\delta_0} + sw_{h+\delta_0}\right)}  \,\lesssim\, \delta_0^{-1+\alpha} \quad \text{ and }\\[3mm]
&\abs{\frac{\partial^2 f}{\partial \nuP\, \partial \boldsymbol{v}}\left((1-s)P_{\delta+h+t} + sw_{h+t}\right)} \,\lesssim\, \frac{(h+t)^\alpha}{(h+t)R(h+t)}.
\end{align*}
Using these estimates and \eqref{eqn:R_Type} in \eqref{eqn:IntInNuP}, we get
\[
\abs{\frac{\partial f}{\partial \boldsymbol{v}}\left((1-s)P_{\delta+h} + sw_{h}\right)} \,\lesssim\, \delta_0^{-1+\alpha}  + \int\limits_0^{\delta_0}\, \frac{(h+t)^\alpha}{(h+t)(h+t)^{1/k_0}}\, \text{d}t\,\lesssim\, h^{\alpha-1/k_0} \,\lesssim\, \frac{h^\alpha}{R(h)}.
\]
Note that the only dependence of the constant in the last inequality on $f$ is on $\left\|f\right\|_\alpha$.
Now, \textit{II} is estimated as follows;
\begin{align*}
\textit{II} &= \mathrlap{\abs{f(P_{\delta +h})-f(w_h)} \le\ \sup\limits_{0\le s\le 1}\, \abs{\frac{\partial f}{\partial \boldsymbol{v}}\left((1-s)P_{\delta+h} + sw_{h}\right)} \cdot \abs{P_{\delta +h} - w_h}}\\[1ex]
&\,\lesssim\, \frac{h^\alpha}{R(h)} \cdot \abs{P_{\delta} - w} = \frac{S(\abs{P_{\delta} - w})^\alpha}{\abs{P_{\delta} - w}} \cdot \abs{P_{\delta} - w} &  (\text{by Proposition \ref{prop:RS_Props}\,\eqref{prop:RS_PropsB}})\\[1ex] 
&= S(\abs{P_{\delta} - w})^\alpha.
\end{align*}
\end{proof}

\section*{Acknowledgements}
I am deeply grateful to Jeffery McNeal for introducing me to this beautiful part of several complex variables, the many discussions regarding this project, and for his generosity and guidance over the years.
I started working on this project as part of my Ph.D. thesis \cite{Rav11} under his guidance at The Ohio State University and the results reported here were obtained during my stay at the Institute for Advanced Study.
I wish to thank the Institute for Advanced Study for their hospitality.
I also wish to thank Ji\v{r}\'{\i} Lebl for insightful discussions regarding this project and Yunus Zeytuncu for his comments on a draft of this article.

\bibliographystyle{amsplain}

\end{document}